\documentclass{modlms}

\usepackage{amssymb,amsmath,amscd,xspace,paralist}
\usepackage{graphicx}

\newtheorem{theorem}{Theorem}[section]
\newtheorem{corollary}[theorem]{Corollary}

\newnumbered{remark}[theorem]{Remark}
\newnumbered{remarks}[theorem]{Remarks}

\def\cF{{\mathcal{F}}}
\def\cE{{\mathcal{E}}}
\def\cC{{\mathcal{C}}}
\def\cH{{\mathcal{H}}}

\def\de#1{\textit{#1}}

\def\R{{\mathbb R}}
\def\Z{{\mathbb Z}}
\def\N{{\mathbb N}}
\def\C{{\mathbb C}}

\def\raw{\rightarrow}

\def\ux{{\underline x}}
\def\uy{{\underline y}}

\def\I{^{-1}}

\def\td{\tilde{d}}
\def\tx{{\tilde{x}}}
\def\tz{{\tilde{z}}}
\def\tf{\tilde{f}}
\def\tg{\tilde{g}}
\def\tM{\tilde{M}}
\def\tE{\tilde{E}}
\def\tPhi{\tilde{\Phi}}
\def\tLambda{\tilde{\Lambda}}

\def\htop{h_{top}}
\def\homeo{homeomorphism}
\def\homeos{homeomorphisms}
\def\BR{boundary retract\xspace}

\DeclareMathOperator{\id}{id}
\DeclareMathOperator{\Per}{Per}
\DeclareMathOperator{\diam}{diam}

\def\barg{\overline{g}}
\def\barf{\overline{f}}

\def\ilim{{\varprojlim}}
\def\bd{{\partial}}

\title[Inverse limits in parameterized families]%
{Inverse limits as attractors in parameterized families}

\date{May 2012}

\author{Philip Boyland, Andr\'e de Carvalho, and Toby Hall}

\classno{54H20 (primary), 37B45, 37E05, 37E10, 37E30, 37E45
  (secondary)}

\extraline{The authors are grateful for the support of FAPESP grants
  numbers~2006/03829-2 and 2010/09667-0}

\begin{document}
\maketitle

\begin{abstract}
We show how a parameterized family of maps of the spine of a manifold
can be used to construct a family of \homeos\ of the ambient manifold
which have the inverse limits of the spine maps as global attractors.
We describe applications to unimodal families of interval maps, to
rotation sets, and to the standard family of circle maps.
\end{abstract}

\section{Introduction}
The use of inverse limits to construct and analyze examples has been
an important tool in dynamical systems since Williams's
work~\cite{williams} on expanding attractors. Given a continuous
self-map $f\colon X\raw X$ of a metric space, its natural extension
$\hat{f}$ is a self-homeomorphism of the inverse limit space
$X_\infty=\ilim(X,f)$. The natural extension $\hat{f}\colon
X_\infty\raw X_\infty$ is, in a precise sense, the dynamically
minimal extension of~$f$ to a homeomorphism.  The space~$X_\infty$ is
defined abstractly as a subspace of $X^\N$, but in many examples the
inverse limit can be embedded inside a manifold~$M$ which has the
original space~$X$ as a spine, and $\hat{f}\colon X_\infty \raw
X_\infty$ can be extended to a self-homeomorphism of~$M$ for which
$X_\infty$ is a global attractor.

In the simplest case, this provides a method for constructing
homeomorphisms of surfaces from endomorphisms of graphs having the
same homotopy type as the surface: the surface homeomorphisms have
attracting sets -- generally with complicated topology -- on which the
dynamics is derived from that of the graph endomorphism. This
construction is useful because it is much easier to construct and
analyze graph endomorphisms than surface homeomorphisms. More
generally, the technique can often be used to embed non-invertible
dynamics as an attractor of a higher-dimensional invertible system.

Barge and Martin~\cite{bargemartin} systematized this idea, describing
a construction to embed the inverse limit of any interval endomorphism
as a global attractor of a plane homeomorphism. As they commented,
their construction readily generalizes to graphs other than the
interval and to higher dimensions, and this generalization is
a special case of the results presented here: a continuous self-map
$f\colon X\raw X$ of a {\em boundary retract}~$X$ of a compact
manifold~$M$ gives rise to an appropriate homeomorphism $\Phi\colon
M\raw M$ provided that~$f$ satisfies a certain topological condition
(it {\em unwraps} in~$M$).

The main purpose of this paper is to develop a parameterized
version of the Barge-Martin construction. Continuously varying
families of maps are of central importance in dynamical systems
theory: apart from their obvious relevance in modelling, one of the
best ways to understand complicated dynamics is to study the way that
it is built up from simple dynamics in parameterized families. The
main result of this paper, Theorem~\ref{BBMparam}, states that the
Barge-Martin construction can be carried out for a parameterized
family~$f_t\colon X\raw X$ in such a way as to yield a continuously
varying family~$\Phi_t\colon M\raw M$ of homeomorphisms, provided
that each~$f_t$ unwraps in~$M$. Under a mild additional assumption
(that there is some~$m>0$ such that $f_t^{m+1}(X)=f_t^m(X)$ for
all~$t$), the attractors~$\Lambda_t$ of~$\Phi_t$ (on which the
dymamics is given by the natural extension of~$f_t$) vary Hausdorff
continuously with the parameter.

The main tool in the Barge-Martin construction is a theorem of Morton
Brown~\cite{brown}, that the inverse limit of a near-homeomorphism of
a compact metric space~$X$ is homeomorphic to~$X$. The parameterized
version of the construction requires this homeomorphism to vary
continuously with the near-homeomorphism, and this extension of
Brown's theorem is presented in Section~\ref{sec:brown}
(Theorem~\ref{brownplus} and Corollary~\ref{ourthing}). The main
theorem of the paper, the parameterized Brown-Barge-Martin (BBM)
construction, is contained in Section~\ref{sec:BBM}.

Section~\ref{apps} contains a brief summary of each of three areas of
application. In Section~\ref{sec:unimodal}, the construction is
applied to the tent family and the quadratic family of unimodal
interval endomorphisms to provide families of homeomorphisms of the
disk~$D^2$ with monotonically increasing dynamics. The inverse limits
of these unimodal maps, their embeddings as attractors of
homeomorphisms, and their relationship to H\'enon maps have been much
studied, and we relate our construction to other work in this area.

The original motivation for this paper was an attempt to understand
the rotation sets that arise in continuously varying families of
homeomorphisms of the torus. This problem, which had previously
resisted analysis, becomes tractable in certain cases when it is
reduced to a one-dimensional problem, allowing it to be attacked using
methods of kneading theory. Section~\ref{sec:rotation} provides a
general description of the relationship between the rotation sets of
the endomorphisms $f_t\colon X\raw X$ and those of the homeomorphisms
$\Phi_t\colon M\raw M$. In Section~\ref{sec:standard}, this is applied
in another example, Arnol'd's {\em standard family} $f_{b,w}$ of
circle endomorphisms, yielding a continuously varying family
$\Phi_{b,w}$ of annulus homeomorphisms with the same rotation sets
as~$f_{b,w}$.

\subsection*{Definitions and notation}
Let~$(X,d)$ and~$(Y,e)$ be compact metric spaces. We write $\cC(X,Y)$
and $\cH(X,Y)$ respectively for the spaces of continuous maps and
\homeos\ $X\raw Y$, endowed with the uniform metric.

A map $f\in\cC(X,Y)$ is a \de{near-homeomorphism} if it is the
uniform limit of homeomorphisms: that is, if it lies in the closure of
$\cH(X,Y)$ in $\cC(X,Y)$. Every near-homeomorphism is onto,
being the uniform limit of continuous surjections from a compact
space. 

Let~$I$ be a compact metric space, which will be considered as a
parameter space. A continuous family $\{f_t\}_{t\in I}$ in~$\cC(X,Y)$
is a \de{near-isotopy} if the map $X\times I \to Y$ given by
$(x,t)\mapsto f_t(x)$ can be uniformly approximated by maps
of the form $(x,t)\mapsto h_t(x)$, where $\{h_t\}_{t\in I}$ is a
continuous family in $\cH(X,Y)$. The term near-isotopy and the
notation~$I$ are intended to suggest the ``standard'' case
$I=[0,1]$.

The question of whether or not a continuous family of
near-homeomorphisms is necessarily a near-isotopy appears to be subtle.
If~$I=[0,1]$ and~$X=Y$ is a compact manifold, then this follows
straightforwardly from the deep result of Edwards and
Kirby~\cite{edwardskirby} that~$\cH(X,X)$ is uniformly locally contractible and
thus  uniformly locally path connected.

Let $\epsilon>0$. A map $g\in\cC(X,Y)$ is called an \de{$\epsilon$-map} if
$g(x_1)=g(x_2)$ implies $d(x_1,x_2)<\epsilon$. The 
set~$\cC_\epsilon(X,Y)$ of continuous $\epsilon$-maps from~$X$ to~$Y$
is an open subset of~$\cC(X,Y)$ since~$X$ is compact.

The set of natural numbers~$\N$ is considered to include~$0$. We will
use a standard metric~$d_\infty$ on the product space~$X^\N$, defined
by
\[d_\infty\left(\ux, \uy\right)=
\max_{i\in\N}\,\,\frac{\min(d(x_i,y_i),1)}{i+1}, \] so that
$d_\infty(\ux,\uy)\le 1/(k+2)$ if $x_i=y_i$ for $0\le i\le k$.  To avoid
excessive notation, we will generally denote the metric on any metric
space indiscriminately by~$d$.

\section{Inverse limits and families of inverse limits}   
\label{sec:brown}
Recall that if $X$ is a metric space and $f\colon X\raw X$ is continuous,
the \de{inverse limit} is the metric space
\begin{equation*}
\ilim(X, f) = \{\ux\in X^\N: f(x_{i+1}) = x_i \ 
\text{for all}\ i\in\N \} \subset X^\N.
\end{equation*}
 Where there is no ambiguity regarding the self-map~$f$, we will use
 the shorter notation $X_\infty$ for $\ilim(X, f)$. The projection
 $\ux\mapsto x_k$ from~$X_\infty$ to the $k^\text{th}$ coordinate is
 denoted $\pi_k\colon  X_\infty\raw X$.

The \de{natural extension} of $f\colon X\raw X$ is the
\homeo\ $\hat{f}\colon X_\infty\raw X_\infty$ defined by \mbox{$\hat{f}(\ux) =
(f(x_0), x_0, x_1, \dots)$}. The self-map~$f$ is semi-conjugate to its
natural extension provided that~$f$ is onto, since $f\circ \pi_0 =
\pi_0\circ\hat{f}$ and~$\pi_0$ is onto if and only if~$f$ is. The
relationship between the dynamics of~$f$, the dynamics of~$\hat{f}$,  and
the topology of~$\ilim(X, f)$ has been much
studied: we refer the reader to the book~\cite{ingrambook} and 
the references therein for more information.
One simple fact that we need here is 
that  $\pi_0$ restricts to a bijection from the set 
of periodic points of~$\hat{f}$ to the set of periodic points of~$f$.

If $f\colon X\raw X$ is a \homeo, then it is clear that $\pi_0\colon X_\infty\to
X$ is a homeomorphism which conjugates~$f$ and~$\hat{f}$. Morton
Brown~\cite{brown} shows that if $X$ is compact and $f$ is a
near-homeomorphism, then $X_\infty$ is still homeomorphic to~$X$.

As noted in the introduction, the parameterized BBM-construction
requires a parameterized version of Brown's Theorem applicable to
near-isotopies. Let~$I$ be a compact parameter space
and~$\{f_t\}_{t\in I}$ be a near-isotopy of a compact metric
space~$X$, and for each~$t\in I$ write $X_\infty^t$ for the inverse limit
$\ilim(X, f_t)$. Brown's theorem provides a homeomorphism $h_t\colon 
X_\infty^t \raw X$ for each~$t$. The parameterized version of Brown's
theorem guarantees that the family~$\{h_t\}$ can be chosen to vary
continuously with~$t$: the easiest way to formulate this is using the
language of fat maps.

Recall that if $\{f_t\}_{t\in I}$ is a continuous family of self-maps
of~$X$, the corresponding \de{fat map} $F\colon X\times I \raw X\times I$ is
defined by
\begin{equation}\label{fat}
F(x,t) = (f_t(x), t).
\end{equation}

A function $G\colon X\times I \raw X\times I$ is called
\de{slice-preserving} if it has the property that
\mbox{$G(X\times\{t\})\subset X\times\{t\}$} for all $t\in I$: we
write $\cC^s(X\times I,X\times I)$ and $\cH^s(X\times I,X\times I)$
for the subsets of slice-preserving elements of $\cC(X\times I,X\times
I)$ and $\cH(X\times I,X\times I)$. Any \mbox{$F\in\cC^s(X\times I,
  X\times I)$} can be written in the form~\eqref{fat}, with $f_t(x)$
the first component of $F(x,t)$.

\begin{remark}
\label{near-isotopy}
It follows immediately from the definitions that the family $\{f_t\}$
in $\cC(X,X)$ is a near-isotopy if and only if its fat map $F\colon X\times
I \raw X\times I$ lies in the closure of \mbox{$\cH^s(X\times I,
  X\times I)$} in $\cC^s(X\times I, X\times I)$.
\end{remark}

The inverse limit $(X\times I)_\infty$ of the fat map $F\colon X\times I\to
X\times I$ provides a natural topology for the family of inverse
limits $X_\infty^t=\ilim(X, f_t)$. Specifically, for each \mbox{$t\in I$}
there is a natural embedding $\iota_t \colon  X_\infty^t \to (X\times
I)_\infty$ given by $\iota_t(x_0, x_1, \ldots) = ((x_0,t), (x_1,t),
\ldots)$; moreover, \mbox{$(X\times I)_\infty$} is the disjoint union of the
subsets $\iota_t(X_\infty^t)$, and
\begin{equation}
\label{CD}
\begin{CD}
X_\infty^t @>\hat{f}_t>> X_\infty^t \\
@VV\iota_t V @VV\iota_t V\\
(X\times I)_\infty @>\hat{F}>> (X\times I)_\infty
\end{CD}
\end{equation}
commutes for each $t\in I$.

The parameterized version of Brown's theorem is:

\begin{theorem}\label{brownplus}
Let $\{f_t\}_{t\in I}$ be a near-isotopy of the compact metric space
$X$, and \mbox{$F\colon X\times I\raw X\times I$} be the corresponding fat
map.  Then for all~$\epsilon>0$ there exists a \homeo\ $\beta\colon 
(X\times I)_\infty\raw X\times I$ such that

\medskip

\begin{enumerate}[(a)]
\item  $\beta\circ \iota_t(X_\infty^t) = X\times
\{t\}$ for all $t\in I$, and
\item $d(\beta, \pi_0) < \epsilon$.
\end{enumerate}
\end{theorem}

\begin{proof}
The proof given here is an adaptation of 
Ancel's short and elegant proof~\cite{ancel} of Brown's Theorem. 

To simplify notation, write $Z = X\times I$ and $Z_\infty = (X\times
I)_\infty$. Observe that for each~$k\in\N$, the projection
$\pi_k\colon Z_\infty \raw Z$ is a $1/(k+2)$-map, since if
$\pi_k(\ux)=\pi_k(\uy)$ then $x_i=y_i$ for $0\le i\le k$.

Let~$\cF=\overline{\cE}$ be the closure in~$\cC(Z_\infty, Z)$ of
\[\cE = \{H\circ \pi_k\,:\,H\in \cH^s(Z,Z) \text{ and }k\in\N
\}.\] Observe that every~$\alpha\in\cF$ is onto (each~$\pi_k$ is onto,
since $F$ is a near-homeomorphism, so that~$\alpha$ is the uniform
limit of onto maps defined on a compact space), and moreover satisfies
$\alpha\circ\iota_t(X_\infty^t) = X\times\{t\}$ for all~$t$ (for
$H\circ\pi_k\circ\iota_t(\ux)=H(x_k,t) \in X\times\{t\}$, and $x_k$
takes every value in~$X$ since~$f_t$ is onto). We will show that
injections (and therefore homeomorphisms which satisfy~(a) of the
theorem statement) are dense in~$\cF$: this will complete the proof,
with (b) following because~$\pi_0\in\cF$.

Given $\delta > 0$,
let $\cF(\delta) = \cF\cap \cC_\delta(Z_\infty, Z)$, an open
subset of~$\cF$. We shall show that it is also dense in~$\cF$. For
this, it suffices to show that for every $H\circ\pi_k\in\cE$ and every
$\eta>0$, there is some $G\in\cF(\delta)$ with $d(G,
H\circ\pi_k)<\eta$. To find such a~$G$, pick $j\ge k$ with
$1/(j+2)<\delta$, and observe that
\[H\circ \pi_k = H \circ F^{j-k} \circ \pi_j.\]
By Remark~\ref{near-isotopy}, $F^{j-k}$ can be approximated
arbitrarily closely by elements of $\cH^s(Z,Z)$, so in particular
there is some $H'\in \cH^s(Z,Z)$ with \mbox{$d(H\circ H'\circ \pi_j, H\circ
\pi_k)<\eta$}. Then $G=H\circ H'\circ\pi_j$ lies in~$\cF(\delta)$
since $\pi_j$ is a $1/(j+2)$-map and $H\circ H'$ is a homeomorphism.

By the Baire category theorem on the complete space~$\cF$, the set
$\bigcap_{n\ge 1}\,\cF(1/n)$ is dense in~$\cF$. However, elements of
this set are injective since they are $1/n$-maps for all $n\ge 1$,
completing the proof that injections are dense in~$\cF$ as required.
\end{proof}

The following corollary will be used in the parameterized Barge-Martin
construction. 

\begin{corollary}\label{ourthing}
 Let~$\{f_t\}_{t\in I}$ be a near-isotopy of the compact metric space~$X$, and
 let the natural extension of $f_t$ to the inverse limit $X_\infty^t$
 be denoted $\hat f_t$. Then for all~$\epsilon>0$ there exist
 homeomorphisms $h_t\colon X_\infty^t \raw X$ for each~$t$ such that

\medskip

\begin{enumerate}[(a)]
\item $h_t\circ \hat{f}_t\circ h_t\I$ is a continuous family of
  homeomorphisms of~$X$, and
\item $d(h_t,\pi_{0,t})<\epsilon$ for all~$t$, where~$\pi_{0,t}\colon X_\infty^t\raw X$
  is projection to the $0^\text{th}$ coordinate.
\end{enumerate}

\medskip

Moreover, if~$X$ is a compact manifold and $\bd X$ is totally
invariant under~$f_t$ for all~$t$, then
\[\bd X_\infty^t = \{\ux\in X_\infty^t\,:\,x_0\in\bd X\}.\]
\end{corollary}

\begin{proof}
Given~$\epsilon>0$, let $\beta \colon  (X\times I)_\infty \to X\times I$ be
a \homeo\ satisfying~(a) and~(b) of Theorem~\ref{brownplus} (and
constructed as in the proof of the theorem). Define
$h_t\colon X_\infty^t\raw X$ by $h_t = p_1\circ\beta\circ\iota_t$, where
$p_1\colon X\times I \raw X$ is projection onto the first coordinate: thus
$\beta\circ\iota_t(\ux) = (h_t(\ux),t)$. Then each~$h_t$ is a
homeomorphism: it is injective and continuous because $\iota_t$,
$\beta$, and $p_1|_{X\times\{t\}}$ are; and it is surjective because
\mbox{$\beta\circ\iota_t(X_\infty^t)=X\times\{t\}$}.  Now
\begin{eqnarray*}
h_t\circ \hat{f}_t\circ h_t\I(x) &=& p_1 \circ (\beta\circ\iota_t)
\circ \hat{f}_t \circ (\beta\circ\iota_t)\I(x,t) \\
&=& p_1 \circ \beta \circ \hat{F}\circ \beta\I(x,t)
\end{eqnarray*}
by~\eqref{CD}, and so depends continuously on~$x$ and~$t$. Hence
$h_t\circ \hat{f}_t\circ h_t\I$ is a continuous family of
homeomorphisms of~$X$ as required.

For~(b), let~$\ux\in X_\infty^t$ and observe that
\[d(h_t(\ux),\pi_{0,t}(\ux)) = d(h_t(\ux), x_0) = d((h_t(\ux),t), (x_0,t)) =
d(\beta\circ\iota_t(\ux), \pi_0\circ\iota_t(\ux))<\epsilon\] since
$d(\beta,\pi_0)<\epsilon$.

For the final statement, if~$X$ is a compact manifold and $\bd X$ is
totally invariant under~$f_t$, then every element $\ux$ of
$X_\infty^t$ either has $x_i\in \bd X$ for all~$i$, or $x_i\not\in\bd
X$ for all~$i$. Now
\[\bd X_\infty^t = h_t\I(\bd X) = \{\ux\in
X_\infty^t\,:\,\beta\circ\iota_t(\ux)\in \bd X\times I\}.\]
Using the notation from the proof of Theorem~\ref{brownplus}, every
$\alpha=H\circ\pi_k\in\cE$ satisfies that \mbox{$\alpha\circ\iota_t(\ux) =
H(x_k,t)$} lies in $\bd X\times I$ if and only if $x_k\in\bd X$, which
establishes the result.
\end{proof}

\section{The parameterized Barge-Martin construction}
\label{sec:BBM}
The parameterized BBM construction starts with a family of maps
defined on a \BR of a compact manifold, with the additional property
that the family unwraps. We begin by defining these and related terms.

Let~$M$ be a compact manifold with non-empty boundary $\bd M$. A
subset~$E$ of~$M$ is said to be a \de{\BR}of~$M$ if there is a 
continuous map $\Psi\colon \bd M\times[0,1]\raw M$ with the following
properties:

\medskip

\begin{enumerate}[(1)]
\item $\Psi$ restricted to $\bd M \times [0, 1)$ is a 
\homeo\ onto $M - E$,
\item $\Psi(\eta, 0) = \eta$, for all $\eta\in\bd M$, and
\item $\Psi(\bd M \times \{1\}) = E$.
\end{enumerate}

\medskip

An alternative characterization is that $\Psi$ decomposes~$M$ into a
continuously varying family of arcs $\{\gamma_\eta\}_{\eta\in\bd M}$
defined by $\gamma_\eta(s)= \Psi(\eta,s)$, whose images are mutually
disjoint except perhaps at their final points. The arc $\gamma_\eta$
has initial point~$\eta$, final point in~$E$, and interior disjoint
from~$E$. Thus, in particular, each point $x\in M-E$ can be written
uniquely as $x=\Psi(\eta,s)$ for some $\eta\in\bd M$ and $s\in[0,1)$.
  The example of interest in the applications described in
  Section~\ref{apps} is when~$E$ is a graph embedded as the spine of a
  surface~$M$ with boundary (Figure~\ref{spine}).

\begin{figure}[htbp]
\begin{center}
\includegraphics[width=0.4\textwidth]{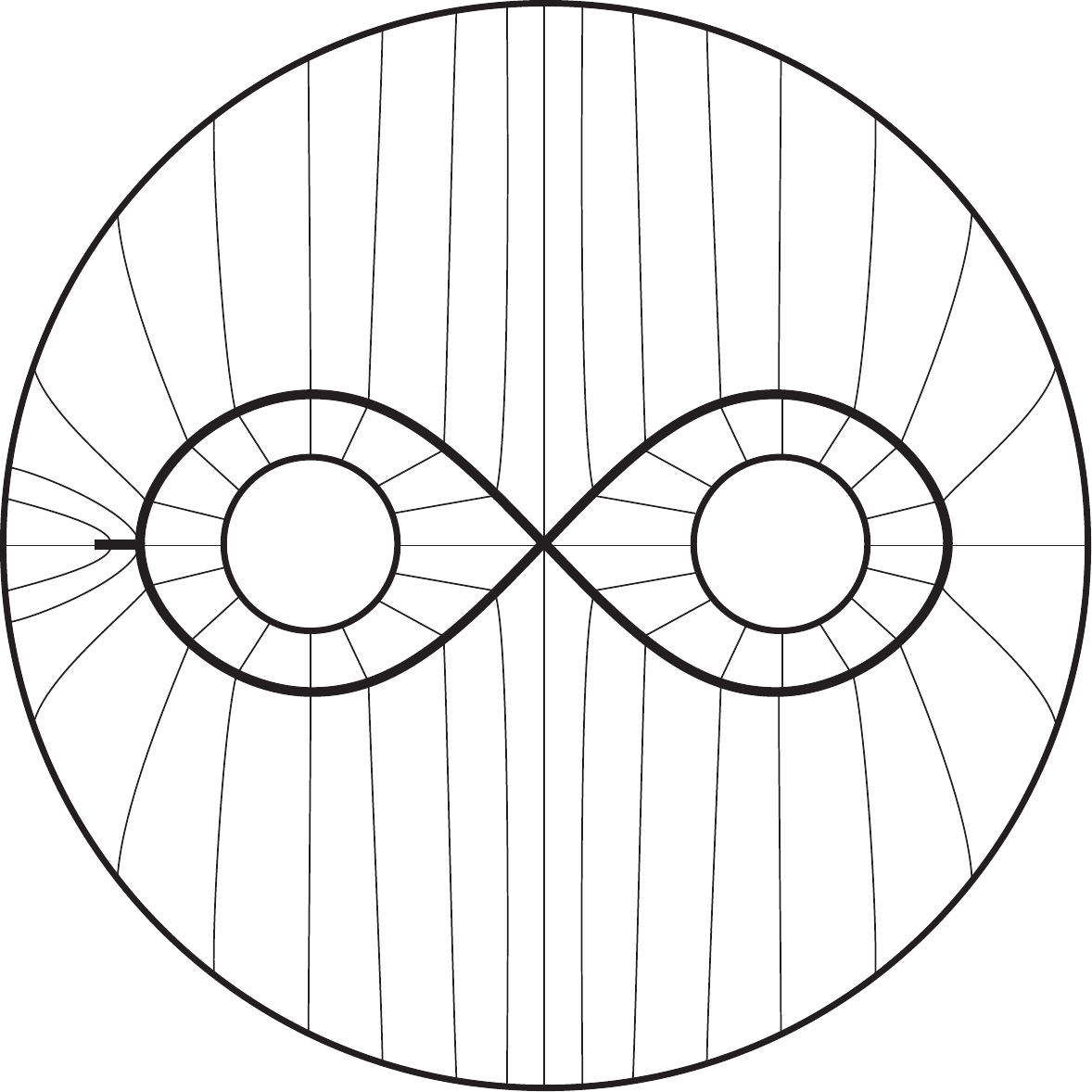}
\caption{The figure 8, with a spike, as a \BR of the pair of
  pants. Each valence~$k$ vertex is the endpoint of~$k$ of the arcs
  $\gamma_\eta$.}
\label{spine}
\end{center}
\end{figure}

Associated to~$\Psi$ is the strong deformation retract
$S\colon M\times[0,1]\raw M$ of $M$ onto~$E$ defined by $S(\Psi(\eta,s),t) =
\Psi(\eta, s+t(1-s))$. The corresponding retraction $R\colon  M\raw E$ is
defined by
\[R(\Psi(\eta,s)) = \Psi(\eta, 1).\]

If~$E$ is a \BR of $M$ with associated retraction $R\colon M\raw E$, then a
continuous map $f\colon E\raw E$ is said to \de{unwrap in $M$} if there is a
homeomorphism $\barf\colon M\raw M$ such that

\medskip

\begin{enumerate}[(1)]
\item $R\circ\barf_{\vert E} = f$, and
\item There is some~$k>0$ such that $\barf^k$ is the identity on~$\bd
  M$. 
\end{enumerate}

\medskip

Such a homeomorphism~$\barf$ is called an \de{unwrapping} of~$f$. A
continuous family $\{f_t\}_{t\in I}$ in $\cC(E,E)$ is said to
\de{unwrap in~$M$} if it has a continuous family $\{\barf_t\}_{t\in I}$ of
unwrappings. 

Finally, we say that~$f\in\cC(X,X)$ \de{stabilizes at
  iterate~$m$} if $f^{m+1}(X)=f^m(X)$ (there is no requirement for~$m$
to be the least such integer). If $f$ stabilizes at iterate~$m$ then
\[\ilim(X, f) = \ilim(f^m(X), f_{\vert f^m(X)}),\]
and the projection $\pi_k\colon \ilim(X, f)\to X$ has image $f^m(X)$ for
all~$k$. 

The parameterized BBM construction can now be stated: the
unparameterized version is obtained on taking the parameter space~$I$
to be a point. Roughly speaking, the theorem states that a continuous
family of continuous self-maps~$f_t$ of~$E$ which unwraps in~$M$ and
stabilizes at a common iterate gives rise to a continuous family of
self-homeomorphisms~$\Phi_t$ of~$M$ having global
attractors~$\Lambda_t$ on which the dynamics of~$\Phi_t$ is conjugate
to that of~$\hat{f}_t$. Moreover, the dynamics of $\Phi_t$
on~$\Lambda_t$ is semi-conjugate to that of $f_t$ on~$\bigcap_{k\ge 0}
f_t^k(E)$ by a semi-conjugacy which can be extended to a continuous
map $M\raw M$ arbitrarily close to the identity; and the
attactors~$\Lambda_t\subset M$ vary continuously with~$t$.

\begin{theorem}
\label{BBMparam}
Let~$M$ be a compact manifold with boundary $\bd M$, $E\subset M$ be a
\BR of~$M$, and $\{f_t\}_{t\in I}$ be a continuous family in
$\cC(E,E)$ which unwraps in~$M$. Suppose moreover that there is
some~$m>0$ such that every~$f_t$ stabilizes at iterate~$m$. Then for
each $\epsilon>0$ there is a continuous family $\{\Phi_t\}_{t\in I}$
in $\cH(M,M)$ such that

\medskip

\begin{enumerate}[(a)]
\item For each~$t\in I$ there is a compact $\Phi_t$-invariant
  subset~$\Lambda_t$ of~$M$ with the following properties:

\medskip

\begin{enumerate}[(i)]
\item $\Phi_t|_{\Lambda_t}\colon \Lambda_t\raw\Lambda_t$ is topologically
  conjugate to $\hat{f}_t\colon \ilim(E, f_t)\raw\ilim(E, f_t)$.
\item If $x\in M-\bd M$, the $\omega$-limit set $\omega(x,\Phi_t)$ is
  contained in~$\Lambda_t$.
\item There is some~$k>0$ such that each $\Phi_t^k$ is the identity
  on~$\bd M$.
\end{enumerate}

\medskip

\item There is a continuous family $\{g_t\}_{t\in I}$ in $\cC(M,M)$
  with $d(g_t, \id)<\epsilon$, $g_t(\Lambda_t)=f_t^m(E)$, and $f_t\circ
  g_t|_{\Lambda_t} = g_t\circ \Phi_t|_{\Lambda_t}$. 
\item For each~$t$, the semiconjugacy $g_t$ restricts to a bijection
  from the set of periodic points of $\Phi_t$ in~$\Lambda_t$ to the
  set of periodic points of~$f_t$.
\item The attractors~$\Lambda_t$ vary Hausdorff continuously with
  $t\in I$. 
\end{enumerate}
\end{theorem}

\begin{proof}
Let $\Psi\colon \bd M \times[0,1]\raw M$ be the map expressing~$E$ as a \BR
of~$M$. Let $\phi\colon [0,1]\raw[0,1]$ be defined by $\phi(s)=2s$ for
$s\in[0,1/2]$ and $\phi(s)=1$ for $s\in[1/2,1]$, and define
$\Upsilon\colon M\raw M$ by $\Upsilon(\Psi(\eta,s)) = \Psi(\eta, \phi(s))$,
which is well defined since $\phi(1)=1$. Because $\phi$ is the uniform
limit of homeomorphisms $[0,1]\raw[0,1]$, the map $\Upsilon$ is a
near-homeomorphism.

Write $N(E)=\Psi(M\times[1/2,1])=\Upsilon\I(E)$: thus $N(E)$ is a
compact neighborhood of~$E$ which is homeomorphic to~$M$, by the
homeomorphism $S\colon  M\raw N(E)$ defined by $S(\Psi(\eta,s))=\Psi(\eta,
(s+1)/2)$.

Let $\{\barg_t\}$ be the unwrapping of the family~$\{f_t\}$, and
define the family $\{\barf_t\}$ in $\cH(N(E), N(E))$ by $\barf_t = S
\circ \barg_t \circ S\I$. Extend $\{\barf_t\}$ to a family in
$\cH(M,M)$ along the arc decomposition given by~$\Psi$: that is, if
$\lambda_t\colon \bd M\raw\bd M$ is the homeomorphism defined by
$\barf_t(\Psi(\eta,1/2)) = \Psi(\lambda_t(\eta), 1/2)$, then
$\barf_t(\Psi(\eta, s))=\Psi(\lambda_t(\eta), s)$ for
\mbox{$s\in[0,1/2]$}.

Now let $H_t = \Upsilon \circ \barf_t\colon  M \raw M$ for each $t\in
I$. Then $\{H_t\}$ is a near-isotopy, since $\Upsilon$ is a
near-homeomorphism and each $\barf_t$ is a homeomorphism. Moreover,
for each~$t\in I$ we have:

\medskip

\begin{enumerate}[(I)]
\item $H_t$ is equal to~$f_t$ on~$E$ (since $\Upsilon$ and $R$
  restrict to the same retraction $\Psi(\eta,s)\mapsto\Psi(\eta,1)$
  of~$N(E)$ onto~$E$, and $R\circ\barf_t|_{E}=f_t$);
\item for all $x\in M-\bd M$, there is some~$n\ge 0$ with $H_t^n(x)\in
  E$ (since $H_t(\Psi(\eta,s))=\Psi(\eta',2s)$ for some~$\eta'$ if
  $s\le 1/2$, and $H_t(\Psi(\eta,s))\in E$ if $s\ge 1/2$); and
\item there is some~$k>0$ such that $H_t^k=\id$ on $\bd M$ (since
  $H_t=\barf_t=\barg_t$ on $\bd M$).

\end{enumerate}

\medskip

Let $h_t\colon M_\infty^t = \ilim(M, H_t) \raw M$ be the homeomorphisms
given by Corollary~\ref{ourthing} (and constructed as in the proof of
the corollary), so that $\{\Phi_t\} = \{h_t\circ \hat{H}_t \circ
h_t\I\}$ is a continuous family in~$\cH(M,M)$, and $d(h_t,
\pi_{0,t})<\epsilon$ for all~$t$. We now show that~(a), (b), (c), and
(d) in the theorem statement hold for this family~$\{\Phi_t\}$.

\medskip\noindent\textit{Property (a):} 
 By~(I) above, there is a homeomorphic copy~$\Omega_t$ of
  $\ilim(E, f_t)$ embedded in $M_\infty^t$, namely
\[\Omega_t = \{\ux\in M_\infty^t\,:\, x_k\in E \text{ for all }k\in\N\},\]
 on which the restriction of~$\hat{H}_t$ is topologically conjugate to
 $\hat{f}_t$. Hence $\Lambda_t = h_t(\Omega_t)$ is a compact
 $\Phi_t$-invariant subset of~$M$ with $\Phi_t|_{\Lambda_t}$
 topologically conjugate to $\hat{f}_t\colon \ilim(E, f_t) \raw \ilim(E,
 f_t)$.

By~(II), $\omega(\ux, \hat{H}_t)\subset \Omega_t$ for all $\ux\in
M_\infty^t$ with $x_0\not\in\bd M$: that is, by the final statement of
Corollary~\ref{ourthing}, for all $\ux\not\in\bd M_\infty^t$. Hence
$\omega(x, \Phi_t) \subset \Lambda_t$ for all~$x\not\in\bd
M$. Similarly, (III) gives that $\hat{H}_t^k(\ux) = \ux$ for all $\ux\in
M_\infty^t$ with $x_0\in\bd M$ (i.e. for all~$\ux\in\bd M_\infty^t$),
from which it follows that $\Phi_t^k$ is the identity on $\bd M$.

\medskip\noindent\textit{Property (b):} 
Let $g_t = \pi_{0,t}\circ h_t\I \colon  M\to M$: that is, using the
  notation of Theorem~\ref{brownplus} and Corollary~\ref{ourthing},
  $g_t(x)=y_0$, where $\beta\I(x,t)=((y_0,t), (y_1,t), \ldots)$. It
  follows from the continuity of~$\beta\I$ that $\{g_t\}_{t\in I}$ is
  a continuous family in $\cC(M,M)$. Moreover, $d(g_t, \id)<\epsilon$
  for all~$t$; and $g_t(\Lambda_t) = g_t\circ h_t(\Omega_t) =
  \pi_{0,t}(\Omega_t) = f_t^m(E)$.  Now
\[g_t\circ\Phi_t = \pi_{0,t}\circ\hat{H}_t\circ h_t\I = H_t\circ
\pi_{0,t}\circ h_t\I = H_t\circ g_t \colon  M \raw M,\]
so that $g_t\circ\Phi_t|_{\Lambda_t} = f_t\circ g_t|_{\Lambda_t}$
by~(I).

\medskip\noindent\textit{Property (c):} 
This is an immediate consequence of the fact that $\pi_{0,t}$
  restricts to a bijection from the set of periodic points of~$\hat
  H_t$ to the set of periodic points of $H_t$, and $\Per(\Phi_t) =
  h_t(\Per(\hat H_t))$. 

\medskip\noindent\textit{Property (d):} 
To show the Hausdorff continuity of $t\mapsto \Lambda_t$, we
  first show that the function $t\mapsto\Omega_t \subset M_\infty^t
  \subset M^\N$ is continuous as a function $I\raw C(M^\N)$ into the
  set of compact subsets of $M^\N$ with the Hausdorff metric. For
  this, it suffices to show that for all~$\delta>0$ there is
  some~$\eta>0$ such that if $d(s,t)<\eta$ then every point of
  $\Omega_t$ is within~$\delta$ of a point of $\Omega_s$.

Recall that all of the~$f_t$ stabilize at iterate~$m$ so that
$f_t\colon f_t^m(E)\raw f_t^m(E)$ is surjective, and 
\[\Omega_t = \{\ux\in M_\infty^t\,:\, x_k\in f_t^m(E) \text{ for all }
k\in\N\}.\] 
Pick $J>1/\delta$ and $\eta>0$ such that
\begin{equation}
\label{uniform}
d(s,t) < \eta \implies d(f_s^j, f_t^j)<\delta \qquad \text{for } 1
\le j < J+m
\end{equation}
using the (uniform) continuity of $t\mapsto f_t^j$ for each~$j$.

Let~$s,t\in I$ with $d(s,t)<\eta$. Let $\ux\in\Omega_t$, so that
$\ux=(x_0,x_1,x_2,\ldots)$ with $x_k \in f_t^m(E)$ and
$f_t(x_{k+1})=x_k$ for all~$k$. Let~$y\in E$ be such that $x_J =
f_t^m(y)$. Define an element \mbox{$\uy=(y_0,y_1,y_2,\ldots)$} of
$(f_s^m(E))^\N$ by setting $y_j = f_s^m(f_s^{J-j}(y))$ for $0\le j\le
J$, and then choosing $y_j\in f_s^m(E)$ for $j>J$ inductively to
satisfy $f_s(y_{j})=y_{j-1}$.

Then $\uy\in\Omega_s$, and $d(\ux, \uy)<\delta$ as required, since 
\[d(x_j,y_j) = d(f_t^{m+J-j}(y), f_s^{m+J-j}(y)) < \delta\qquad  \text{ for
}  0\le j\le J \]
by~\eqref{uniform}, and $1/(j+2) < \delta$ for $j>J$.

To complete the proof that $t\mapsto\Lambda_t$ is continuous, recall
(using the notation of the proof of Corollary~\ref{ourthing}) that
$\Lambda_t = h_t(\Omega_t) =
p_1\circ\beta\circ\iota_t(\Omega_t)\subset M$. It is required to show
that for all~$\delta>0$ there is some~$\eta>0$ such that, if
$d(s,t)<\eta$, then every point of~$\Lambda_t$ is within~$\delta$ of a
point of~$\Lambda_s$. Now the map
\[K\colon \{(\ux,t)\in M^\N\times I \,:\, \ux\in\Omega_t\} \to (M\times
I)_\infty\]
defined by $K(\ux,t)=\iota_t(\ux)$ is continuous, so we can pick
$\xi>0$ such that if $d((\ux,t),(\uy,s))<\xi$ then
$d(p_1\circ\beta\circ K(\ux,t), p_1\circ\beta\circ K(\uy, s)) <
\delta$. By the first part of the proof, there is some~$\eta>0$ such
that if $d(s,t)<\eta$ and $\ux\in\Omega_t$, then there is some
$\uy\in\Omega_s$ with $d((\ux,t), (\uy, s)) < \xi$. 

Then, given $s,t\in I$ with $d(s,t)<\eta$ and $x\in\Lambda_t$, write
$x=p_1\circ\beta\circ K(\ux,t)$ for some $\ux\in\Omega_t$. Let
$\uy\in\Omega_s$ with $d((\ux,t),(\uy,s))<\xi$: then
$y=p_1\circ\beta\circ K(\uy,s)$ lies in~$\Lambda_s$ with
$d(x,y)<\delta$ as required.

\end{proof}

\begin{remarks}
\label{BBM-rks}
\begin{enumerate}[(a)]
\item The proof of part~(a) of the theorem does not require the family to
stabilize uniformly. 
\item If condition~(2) in the definition of an unwrapping is removed
  (so that an unwrapping $\barf\in\cH(M,M)$ is only required to
  satisfy $R\circ \barf_{\vert E}=f$), then the theorem still holds
  except for part~(a)(iii).
\item By part (a)(ii), if each boundary component of~$M$ is collapsed
  to a point then each~$\Phi_t$ becomes a homeomorphism of the
  resulting space (which in the surface case is a surface without
  boundary), for which every point except for a finite number of
  repelling periodic points has $\omega$-limit set contained
  in~$\Lambda_t$.
\end{enumerate}
\end{remarks}

\section{Applications}
\label{apps}

\subsection{Unimodal dynamics}
\label{sec:unimodal}
The inverse limits of members of unimodal families of interval
endomorphisms such as the tent family
\begin{equation*}
T_s(x) =  \min\{sx, s(1 - x)\}, \qquad 1\le s\le 2 
\end{equation*}
and the quadratic family
\[
f_a(x) = a-x^2, \qquad -1/2 \le a\le 2
\]
have been intensively studied: a major recent advance is the proof by
Barge, Bruin, and {\v S}timac of Ingram's conjecture, that if \mbox{$1
  \leq t < s \leq 2$} then the inverse limits $\ilim(I, T_t)$ and
$\ilim(I, T_s)$ are not homeomorphic (see \cite{bargebruinstimac} and
the additional references therein).  

Barge and Martin~\cite{bargemartin} showed that any map of the
interval unwraps in the disk, and their construction gives a
continuous family of unwrappings of any unimodal family. Thus
Theorem~\ref{BBMparam} provides a continuously varying family
\mbox{$\Phi_s:D^2\raw D^2$} of \homeos\ of the disk, with $\Phi_s$
having $\ilim(I, T_s)$ as a global attracting set. Moreover, since the
family $T_s$ uniformly stabilizes at iterate one, these attractors
vary continuously in the Hausdorff topology, although no two are
homeomorphic. The family $\Phi_s$ has dynamics which increases
monotonically with~$s$: for example, the number of periodic orbits of
each period increases with~$s$, as does the topological entropy
$\htop(\Phi_s)=\htop(f_s)=\log(s)$.

Extending the range of parameters to $s\in [0,2]$ illustrates the need
for the stabilization hypothesis in Theorem~\ref{BBMparam}. For
$s\in[0,1)$ we have $T_s^k([0,1])\searrow \{0\}$ as $k\raw\infty$,
so these $T_s$ never stabilize and $\ilim(I, T_s)$ is a
point. On the other hand, $T_1$ is the identity on
$T_1([0,1])=[0,1/2]$, so $T_1$ stabilizes at iterate~1 and $\ilim(I,
T_1)$ is an arc. Hence there is a discontinuous change
in the attractor at~$s=1$.

Theorem~\ref{BBMparam} may likewise be applied to the quadratic family,
yielding a family of of plane homeomorphisms with monotonically
increasing dynamics.

All of the constructions in this paper are strictly in the
$C^0$-category.  Embeddings of the inverse limits of members of the
tent family as attractors of \homeos\ with varying degrees of
increased regularity have been carried out
in~\cite{mis,szczechla,barge,bruin}.  The fascinating question of
which inverse limits and families of inverse limits can be embedded as
attractors for $C^r$-diffeomorphisms is little understood. Barge and
Holte~\cite{bargeholte} show that for certain parameter ranges the
real H\'enon attractor is conjugate to an inverse limit from the quadratic
family. For example, for hyperbolic values of~$a$ and small
enough~$b$, the H\'enon map $H_{a,b}$ restricted to its attractor is
conjugate to the inverse limit $\ilim(I,f_a)$. As has been pointed out
by several authors (for example~\cite{HolWhi,HolWil}), this cannot be
extended to a uniform band of small values of $b$ in parameter space,
since bifurcation curves for periodic orbits in the H\'enon family
cross arbitrarily close to~$b=0$. In fact, the ``antimonotonicity"
results of~\cite{KKY} show that even if it were possible to embed the
inverse limit of the unimodal family as a family of diffeomorphisms of
the plane, it is not possible to do so with high regularity. Much of
the work done to understand infinitely renormalizable H\'enon
maps~\cite{dCLM,LyuMar2} is based on comparisons between them and the
inverse limits of infinitely renormalizable unimodal maps.

The relationship of the inverse limit of the complex quadratic family
to the complex H\'enon family was used to great advantage
in~\cite{HOV1,HOV2}.  More generally, the dynamics of families of
homeomorphisms of $\R^2$ and $\C^2$ are often related to the inverse
limits of 1-dimensional endomorphisms of graphs and ($\R$-)trees and
to branched surfaces (see~\cite{eqgpa}).

\subsection{Rotation sets}
\label{sec:rotation}
If $\Phi\colon M\raw M$ is a map of a compact, smooth manifold which acts as
the identity on $H_1(M;\R)$, one may define the rotation vector under
$\Phi$ of a point $x\in M$. The usual way to do this is to lift $\Phi$
to $\tPhi$ on the universal free Abelian cover $\tM_A$, and then use
the lifts of a basis of closed one-forms to measure the average
displacements of points $\tx\in\tM_A$ under $\tPhi$ (see, for example,
\S3.3 
of~\cite{bdabelian}).  Because $\Phi$ acts as the identity on
$H_1(M;\R)$, $\tPhi$ commutes with the deck transformations of
$\tM_A$, and so the rotation vector is independent of the choice of
lift $\tx\in\tM_A$ of a point $x\in M$.  However, the rotation vector
does depend in a simple way on the choice of lift $\tPhi$. After
fixing a lift $\tPhi$ one defines the rotation set $\rho(\tPhi)$ (or
$\rho(\Phi)$ if a lift is understood) as the collection of rotation
vectors of points $\tx\in\tM_A$.

Now consider a self-map $f\colon E\raw E$ of a \BR $E$ of~$M$.  Assume that
$f$ acts as the identity on $H_1(E;\R)$ and unwraps to a homeomorphism
\mbox{$\barf\colon M\raw M$} with $\barf_{\vert \bd M} = \id$. Let
$\Phi\colon M\raw M$ be the BBM-extension \homeo\ obtained from
Theorem~\ref{BBMparam} by taking the parameter space~$I$ to be a point
and $\epsilon < \diam(M)/10$. Since $E$ is a strong deformation
retract of $M$, we also have that $\Phi_* = \id$ on $H_1(M;\R)$. Thus
both the rotation sets $\rho(\Phi)$ and $\rho(f)$ can be defined. They
are essentially the same, provided that care is taken to choose
corresponding lifts as will now be described.

By construction $\Phi_{\vert \bd M} = \id$, and we choose the lift
$\tPhi$ to $\tM_A$ which has \mbox{$\tPhi_{\vert \bd \tM_A} = \id$}.
Now we also lift the map $g\colon M\raw M$ with $\td(g,\id) < \epsilon$
given by Theorem~\ref{BBMparam}(b) to a $\tg\colon \tM_A\raw\tM_A$ with
$\td(\tg,\id) < \epsilon$.  If $\tE\subset\tM_A$ and
$\tLambda\subset\tM_A$ are the lifts of $E$ and $\Lambda$ inside the
universal free Abelian cover, we choose the lift $\tf\colon \tE\raw\tE$ of
$f$ so that $\tf\circ\tg(\tz) = \tg \circ \tPhi(\tz)$ for all
$\tz\in\tLambda$. Now the rotation set of a map is the same as its
rotation set restricted to the recurrent set. Since by the BBM
construction the recurrent set of $\Phi$ is contained in $\Lambda\cup
\bd M$, we have $\rho(\tPhi) = \rho(\tPhi_{\vert \tLambda}) \cup
\{0\}$.  Finally, since $\td(\tg,\id) < \epsilon$ we get that for any
$\tz\in\tLambda$, $d(\tPhi^k(\tz), \tf^k(\tg(\tz))) < \epsilon$ for
all $k\in\Z$. Thus since $\tg(\tLambda) = \tE$, we have
$\rho(\tPhi_{\vert \tLambda}) = \rho(\tf)$, and so
\begin{equation}\label{rot}
\rho(\tPhi) = \rho(\tf) \cup \{0\}.
\end{equation}

By Theorem~\ref{BBMparam}, \eqref{rot} also holds for parameterized
families and the attractors in the BBM-extension family vary
continuously provided that the family stabilizes uniformly. An
application is given in the authors' forthcoming paper {\em New
  rotation sets in a family of torus homeomorphisms}, in which we
consider a family of maps~$k_t$ of the figure eight $E$ embedded as
the spine of the two torus minus a disk. The resulting BBM-extension
family is extended to a family of torus \homeos~$K_t$ as in
Remark~\ref{BBM-rks}~(c). The intricate sequence of bifurcations of the
rotation sets of the family~$k_t$ can be described in detail using
kneading theory techniques, and by~\eqref{rot} the rotation sets of
the family~$K_t$ of torus homeomorphisms undergo the same
bifurcations.

\subsection{The standard family of circle maps}
\label{sec:standard}
The standard family of degree-one circle maps was introduced by
Arnol'd~\cite{arnold}. This two-parameter family
\mbox{$f_{b,\omega}\colon S^1\raw S^1$}, for $b\geq 0$ and $0 \leq \omega
\leq 1$, is defined via its lifts $\tf_{b,\omega}\colon \R\raw\R$, which are
given by
\begin{equation*}
\tf_{b,\omega}(x) = x + \omega + \frac{b}{2\pi}\sin(2 \pi x).
\end{equation*}
The dynamics and bifurcations of this family have been much studied
(see for example~\cite{bdcircle,brucks,rempe}).  The main objects of
interest in its parameter space are the so-called Arnol'd tongues
$T_r$ given by
\begin{equation*}
T_r = \{(b,\omega) \,:\, r \in\rho(f_{b,\omega})\}.
\end{equation*}

The core circle of an annulus $A$ is a \BR of~$A$, and the family
$f_{b,\omega}$ unwraps in~$A$.  Further, since each $f_{b,\omega}$ is
onto, the family uniformly stabilizes at iterate zero. Thus, after
restricting to the compact parameter space $b\in[0,b^*]$ for some
large~$b^*$, Theorem~\ref{BBMparam} yields a continuous family of
annulus \homeos\ $\Phi_{b,\omega}$ each having an attracting set
$\Lambda_{b,\omega}$ homeomorphic to $\ilim(S^1, f_{b,\omega})$ and,
by \eqref{rot}, $\rho(\Phi_{b,\omega}) = \rho(f_{b,\omega}) \cup
\{0\}$.

By composing each $\Phi_{b,\omega}$ with an appropriate lateral push
on and near on the boundary we can obtain a new family of
\homeos\ $\Phi_{b,\omega}^0$ such that the rotation numbers of each
$\Phi_{b,\omega}^0$ restricted to $\bd A$ are contained in
$\rho(f_{b,\omega})$. We then have that
\begin{equation*} 
\rho(\Phi_{b,\omega}^0) = \rho(f_{b,\omega}) 
\end{equation*}
for all $(b,\omega)\in[0,b^*]\times[0,1]$.

When $b\leq 1$, each $f_{b,\omega}$ is a \homeo\ and so the attractor
of $\Phi_{b,\omega}^0$ is an invariant circle on which the dynamics is
topologically conjugate to $f_{b,\omega}$. When $b> 1$, however, the
Arnol'd tongues $T_r$ for rational $r$ begin to overlap and for
irrational $r$, $T_r$ opens from a Lipschitz curve into a
tongue. These changes are accompanied by an elaborate sequence of
bifurcations which are all shared by the family $\Phi_{b,\omega}^0$.

\begin{acknowledgements} 
We would like to thank Marcy Barge for useful conversations.
\end{acknowledgements}

\bibliography{familyrefs}

\affiliationone{
Philip Boyland\\
Department of Mathematics \\
University of Florida \\
372 Little Hall\\
Gainesville, FL 32611-8105, USA
\email{boyland@math.ufl.edu}
}
\affiliationtwo{
Andr\'e de Carvalho\\
Departamento de Matem\'atica Aplicada\\
IME-USP\\
Rua Do Mat\~ao 1010\\
Cidade Universit\'aria\\
05508-090 S\~ao Paulo SP, Brazil
\email{andre@ime.usp.br}
}
\affiliationthree{
Toby Hall\\
Department of Mathematical Sciences\\
University of Liverpool\\
Liverpool L69 7ZL, UK
\email{tobyhall@liv.ac.uk}
}

\end{document}